\documentclass[12pt]{amsart}
\usepackage{times,fullpage}
\usepackage{latexsym,amscd,amssymb}
\newtheorem{thm}{Theorem}
\newtheorem{prop}[thm]{Proposition}
\newtheorem{lem}[thm]{Lemma}
\newtheorem{cor}[thm]{Corollary}

\theoremstyle{remark}

\newtheorem{ex}{Example}

\theoremstyle{definition}

\newcommand{\C}{\mathbb{ C}}

\newcommand{\Q}{\mathbb{ Q}}
\newcommand{\R}{\mathbb{ R}}

\newcommand{\Z}{\mathbb{ Z}}

\title{Pontryagin numbers and nonnegative curvature}
\author{D.~Kotschick}
\address{Mathematisches Institut, {\smaller LMU} M\"unchen,
Theresienstr.~39, 80333~M\"unchen, Germany}
\email{dieter@member.ams.org}
\thanks{The author gratefully acknowledges the support of The Bell Companies Fellowship at the Institute for Advanced Study in Princeton.}
\date{\today; \copyright{\ D.~Kotschick 2009}}
\subjclass[2000]{primary 53C21, 57R20; secondary 53C23, 57R75}

\begin{document}

\begin{abstract}
We prove that any rational linear combination of Pontryagin numbers that is not a multiple of the signature is unbounded on 
connected closed oriented manifolds of nonnegative sectional curvature. 
Combining our result with Gromov's finiteness result for the signature yields a new characterization of the $L$-genus.
\end{abstract}

\maketitle


\section{Introduction}

This paper is concerned with the boundedness of Pontryagin numbers, and of their linear combinations, for
connected closed oriented manifolds of nonnegative sectional curvature. By Thom's results~\cite{T} on cobordism 
theory, the signature is a rational linear combination of Pontryagin numbers. The following is our main result:
\begin{thm}\label{t:main}
Any rational linear combination of Pontryagin numbers that is not a multiple of the signature 
is unbounded on connected closed oriented manifolds of nonnegative sectional curvature. 
\end{thm}
Combined with Gromov's celebrated Betti number bound~\cite{G}, which implies a bound on the signature,
this theorem gives the following differential-geometric characterization of the signature:
\begin{cor}\label{c}
A rational linear combination of Pontryagin numbers is bounded on connected closed oriented manifolds of 
nonnegative sectional curvature if and only if it is a multiple of the signature. 
\end{cor}
This corollary bears some resemblance to the characterization of the $\hat A$-genus as the only 
linear combination of Pontryagin numbers that vanishes on all closed oriented spin manifolds of positive 
scalar curvature; see Gromov--Lawson~\cite[Corollary B]{GL}. 
Nevertheless, the two statements are quite different. On the one hand, positive scalar curvature manifolds define
an ideal in the spin cobordism ring, and this ideal is identified with the kernel of the $\hat A$-genus using
the Atiyah--Singer theorem and the Lichnerowicz argument. On the other hand, to admit a metric of nonnegative sectional curvature 
is not a cobordism property, and manifolds with such a metric do not define an ideal in the oriented cobordism group.
Furthermore, the signature of nonnegatively curved manifolds does not always vanish, as shown by the classical
examples like $\C P^{2k}$, but the boundedness of the signature follows from a ``soft'' argument~\cite{G}. Of course,
to obtain meaningful statements about the boundedness of a nonvanishing cobordism invariant one has to 
restrict to connected manifolds.

Setting aside the curvature condition in Theorem~\ref{t:main} and Corollary~\ref{c}, we obtain a purely
topological result:
\begin{cor}\label{c2}
A rational linear combination of Pontryagin numbers can be bounded in terms of Betti numbers if and only if it is a multiple of the signature. 
\end{cor}

For the proof of Theorem~\ref{t:main} we shall construct sequences of ring generators for the rational cobordism ring 
$\Omega_{\star}\otimes\Q$ with the property that each generator of dimension $\geq 8$ belongs to a family of possible choices 
on which a certain indecomposabe Pontryagin number is unbounded. These families of examples are $\C P^{2k}$-bundles 
over $S^4$, which have metrics of nonnegative sectional curvature by the results of Grove and Ziller~\cite{GZ}. The case of 
$\C P^2$-bundles was previously considered by Dessai and Tuschmann~\cite{DT}, who essentially proved Theorem~\ref{t:main} 
in dimension $8$. More generally, they proved that in every dimension of the form $4n>4$ there is some Pontryagin number 
that is unbounded on connected manifolds of nonnegative sectional curvature.
Following Borel and Hirzebruch~\cite{BH}, Dessai and Tuschmann~\cite{DT} used Lie theory to calculate a Pontryagin
number by what Weil is said to have called ``digging roots and lifting weights''. In contrast with this, we do not use Lie theory
in this paper. All our calculations are straightforward applications of Grothendieck's 
definition of Chern classes and of Hirzebruch's splitting principle. These calculations actually give a constructive 
proof of Corollary~\ref{c}, allowing us to determine the signature as a linear combination of Pontryagin numbers.
An interesting feature of this argument is that it is independent of the Hirzebruch signature theorem. We shall elaborate on
this fact in the final section of the paper.

\section{Projective space bundles over the four-sphere}

We consider complex vector bundles $E\longrightarrow S^4$ of rank $2k+1$ with $k\geq 1$.
These bundles are classified by their second Chern classes $c_2(E)\in H^4(S^4)$. Fixing
an orientation class $x\in H^4(S^4)$, we write $c_2(E)=c\cdot x$ with $c\in\Z$.

Every rank $(2k+1)$-bundle $E$ over $S^4$ is isomorphic to the direct sum of a rank $2$
bundle with the same Chern class and a trivial bundle of rank $2k-1$. Therefore, the structure group of $E$
reduces to $SU(2)$. In other words, $E$ is associated to a principal $SU(2)$-bundle $P\longrightarrow S^4$.

Let $X=X_c$ be the projectivisation of a vector bundle $E_c$ with Chern class $c_2(E)=c\cdot x$. This is a
$\C P^{2k}$-bundle $\pi\colon X_c\longrightarrow S^4$, also with structure group $SU(2)$. Thus
$$
X_c = (P_c \times \C P^{2k})/SU(2) \ ,
$$
where $P_c$ is the principal $SU(2)$-bundle over $S^4$ with second Chern number $c$. Here $SU(2)$
acts on $\C P^{2k}$ via the block inclusion $SU(2)\subset SU(2k+1)$.
It was proved by Grove and Ziller~\cite{GZ} that $P_c$ admits an $SU(2)$-invariant metric of nonnegative 
sectional curvature. The non-decreasing property of curvature in submersions implies that the product of the 
Grove--Ziller metric on $P_c$ with the Fubini--Study metric on $\C P^{2k}$ induces a nonnegatively curved
metric on $X_c$. Thus we have the following:
\begin{thm}\label{t:GZ}{\rm (cf.~\cite{GZ})}
The projectivisation $X_c$ of any complex vector bundle over $S^4$ admits a Riemannian metric of nonnegative sectional curvature.
\end{thm}

The cohomology ring of $X_c$ is described by the following:
\begin{prop}\label{p:coho}
The cohomology ring of $X_c$ is generated by two classes $x\in H^4(X_c)$ and $y\in H^2(X_c)$ subject to the 
relations $x^2=0$ and $y^{2k+1}+cxy^{2k-1}=0$. The signature of $X_c$ vanishes.
\end{prop}
\begin{proof}
Consider the fibration $\pi\colon X_c\longrightarrow S^4$. By the Leray--Hirsch theorem, $\pi^*$ is injective in
cohomology, and we identify $x$ and $\pi^*x$. Clearly $x^2=0$.

The cohomology ring $H^*(X_c)$ is generated as a $H^*(S^4)$-module by a class $y\in H^2(X_c)$ which restricts 
to every fiber as a generator of $H^2(\C P^{2k})$.
Grothendieck's definition of Chern classes defines $c_2(E_c)$, which by assumption equals $cx$, by the 
equation $y^{2k+1}+c_2(E_c)y^{2k-1}=0$.

Now, the middle-dimensional cohomology $H^{2k+2}(X_c)$ is spanned by $y^{k+1}$ and $xy^{k-1}$. The latter
has square zero because $x^2=0$, and so spans an isotropic subspace for the intersection form. Thus the 
intersection form is indefinite of rank $2$. This implies the vanishing of the signature.
\end{proof}

To compute the Pontryagin classes of $X_c$ note that $TX_c=T\pi\oplus\pi^*TS^4$, where $T\pi$ is the tangent
bundle along the fibers. As the tangent bundle of $S^4$ is stably trivial, the Pontryagin classes of $X_c$ are the 
Pontryagin classes of $T\pi$. This latter bundle actually has a complex structure, and we can write down its Chern
classes:
\begin{prop}\label{p:Chern}
The total Chern class of $T\pi$ is $c(T\pi)=(1+y)^{2k+1}+cx(1+y)^{2k-1}$.
\end{prop}
Note that $T\pi$ has complex rank $2k$, and so its top Chern class is $c_{2k}(T\pi)\in H^{4k}(X_c)$. The above polynomial
has a component also in degree $H^{4k+2}(X_c)$, but this component is $y^{2k+1}+cxy^{2k-1}$, which vanishes by Proposition~\ref{p:coho}.
\begin{proof}
Consider the exact sequence
$$
1\longrightarrow L^{-1}\longrightarrow \pi^*E\longrightarrow L^{-1}\otimes T\pi \longrightarrow 1\ ,
$$
where $L$ is the fiberwise hyperplane bundle on the total space. Tensoring with $L$, we conclude $c(T\pi)=c(L\otimes \pi^*E)$.
As $c(L)=1+y$ and $c(\pi^*E)=1+cx$, the claim follows.
\end{proof}
Let 
$$
c(T\pi)=(1+y)^{2k+1}+cx(1+y)^{2k-1} = (1+y)^{2k-1}(1+y_1)(1+y_2) 
$$
with $(1+y_1)(1+y_2) = 1+2y+y^2+cx$ a formal factorization in the sense of the splitting principle. Then the total
Pontryagin class of $X_c$ is given by
\begin{equation}\label{eq:tp}
p(TX_c)=p(T\pi) = (1+y^2)^{2k-1}(1+y_1^2)(1+y_2^2) \ .
\end{equation}
Any top-degree polynomial in the Pontryagin classes is obviously a linear combination of $y^{2k+2}$ and $cxy^{2k}$.
The relation $y^{2k+2}=-cxy^{2k}$ from Proposition~\ref{p:coho}, together with $\langle xy^{2k},[X_c]\rangle =1$ implies:
\begin{prop}\label{c:mult} 
All Pontryagin numbers of $X_c$ are universal multiples of $c \ $.
\end{prop}

As for precise values, for our proof of Theorem~\ref{t:main} we only need the non-vanishing of one 
particular Pontryagin number for each $X_c$ with $c\neq 0$.
Let $M$ be a closed oriented manifold of dimension $4n$ with total Pontryagin class
$$
p(TM)=\prod_i (1+y_i^2) \  .
$$
Then the number $s_n(M)$, introduced by Thom~\cite{T}, is defined by 
$$
s_n(M)  = \sum_i \langle y_i^{2n},[M]\rangle \ . 
$$
The splitting principle implies that this is a linear combination of Pontryagin numbers.

\begin{prop}\label{p:s}
The manifolds $X_c$ satisfy $s_{k+1}(X_c) = -(2k+1)(2k+3)\cdot c \ $. In particular $s_{k+1}(X_c)$ is non-zero as soon as $c$ is.
\end{prop}
This is in fact a special case of a result of Borel and Hirzebruch~\cite[Lemma 28.2]{BH}.
\begin{proof}
Since the tangent bundle of $S^4$ is stably trivial, we have $p(TX_c)=p(T\pi)$, and the latter is given by~\eqref{eq:tp}.
To calculate $s_{k+1}(X_c)$ we have to evaluate $(2k-1)y^{2k+2}+y_1^{2k+2}+y_2^{2k+2}$ on the fundamental class of $X_c$.
For this we need the following formula, which is easily proved by induction on $n$ using the defining equations $y_1+y_2=2y$ and 
$y_1y_2=y^2+cx$:
\begin{equation}\label{eq:lem}
y_1^{2n+2}+y_2^{2n+2} = 2y^{2n+2}-2c(n+1)(2n+1)xy^{2n} \ .
\end{equation}
Using this formula with $n=k$ and recalling the relation $y^{2k+2}=-cxy^{2k}$ in the cohomology ring of $X_c$ we find
$$
(2k-1)y^{2k+2}+y_1^{2k+2}+y_2^{2k+2}  = -(2k+1)(2k+3) \cdot c \cdot xy^{2k} \ .
$$
This completes the proof.
\end{proof}

We can now prove Theorem~\ref{t:main} using cobordism theory. This theory implies that the signature equals some linear 
combination of Pontryagin numbers. In dimension $4i$ let $L_i$ be this linear combination.

\begin{proof}[Proof of Theorem~\ref{t:main}]
It was proved by Thom~\cite{T} that the rational oriented cobordism ring $\Omega_{\star}\otimes\Q$ is a polynomial ring with one generator
$\alpha_i$ in dimension $4i$ for $i=1,2,\ldots$. Thom also showed that the cobordism class of a closed oriented $(4i)$-manifold $M$ can be taken as a 
generator $\alpha_i$ if and only if $s_i(M)\neq 0$. Thus, using Proposition~\ref{p:s}, we may take as generators the following cobordism classes:
$\alpha_1 = [\C P^2]$ and $\alpha_i=[X_c]$ with $X_c$ a $\C P^{2i-2}$-bundle over $S^4$ for $i\geq 2$, with any $c\neq 0$.

As there is nothing to prove in dimension $4$, where the only Pontryagin number is a multiple of the signature, we assume now that
$i\geq 2$. A $\Q$-vector space basis for $\Omega_{4i}\otimes\Q$ is given by the elements $\alpha_I = \alpha_{i_1}\cdot\ldots\cdot\alpha_{i_l}$,
where $I=(i_1,\ldots,i_l)$ ranges over all partitions of $i$. Among these basis vectors there is one,
namely $\alpha_{(1,\ldots,1)}=[\C P^2\times\ldots\times\C P^2]$, with non-zero signature. All the other basis vectors are in the codimension one 
subspace $\ker (L_i)\subset\Omega_{4i}\otimes\Q$, because each of them contains at least one $X_c$, and these have vanishing signature
by Proposition~\ref{p:coho}. (Recall that the signature is multiplicative in products by the K\"unneth formula.)
Thus the $\alpha_I$ with $I\neq (1,\ldots,1)$ form a vector space basis for $\ker (L_i)$.

Let $f\colon \Omega_{4i}\otimes\Q\longrightarrow\Q$ be any linear combination of Pontryagin numbers. If $f$ is not a multiple of $L_i$, then
$\ker (f)\cap\ker (L_i)$ is a proper subspace of $\ker (L_i)$. It follows that one of the $\alpha_I$ with $I\neq (1,\ldots,1)$ is not in $\ker (f)$, i.e.
$f(\alpha_I)\neq 0$ for this particular $I$. 

If $I=(i_1,\ldots,i_l)$, let $n>0$ be the number of $i_j$ which are $\neq 1$. By Proposition~\ref{c:mult}, $f(\alpha_I)$, which is a linear
combination of products of Pontryagin numbers of the factors $\alpha_{i_j}$, is a multiple of $c^n$. 
It is a non-zero multiple, and is therefore unbounded as we vary $c$. But the corresponding product of copies of $\C P^2$ and the various 
$X_c$ is connected and has a metric of nonnegative sectional curvature by Theorem~\ref{t:GZ}. This completes the proof.
\end{proof}
This proof also yields yet another characterization of the signature:
\begin{cor}
Any linear combination of Pontryagin numbers that vanishes on the total spaces of $\C P^{2k}$-bundles over $S^4$ is a multiple of the signature.
\end{cor}
This is related to the characterization of the signature by its multiplicativity property for fiber bundles due to Borel and Hirzebruch~\cite[Theorem 28.4]{BH}.
The vanishing for $\C P^{2k}$-bundles over $S^4$ is a very special instance of the mutiplicativity. Note that we did not use the signature theorem, or the 
multiplicativity of the signature, other than for the bundles for which we proved it directly in Proposition~\ref{p:coho}.

\section{Relation with the signature theorem}

As we mentioned earlier, it follows from Thom's results~\cite{T} that the signature of a closed oriented $(4i)$-manifold is a linear combination of Pontryagin 
numbers; see Hirzebruch~\cite{TMAG}. The Hirzebruch signature theorem is a much more precise statement, in that it identifies this linear combination $L_i$
as the degree $i$ component of the $L$-genus associated with the power series $Q(x)=x/tanh(x)$; see again~\cite{TMAG}.

The arguments of this paper are independent of the Hirzebruch signature theorem. We have constructed vector space bases for the kernel of the 
signature in $\Omega_{4i}\otimes\Q$ depending on a non-zero integer $c$. In fact, we could introduce several constants, a different one for each $i\geq 2$,
but it is not clear whether this is useful. In any case, the signature is, up to normalization, the only linear combination of Pontryagin numbers which vanishes
on our explicit basis manifolds. Therefore, we can use these bases to determine $L_i$, without using the Hirzebruch signature theorem.
We shall illustrate this in small dimensions, by carrying out some explicit calculations of Pontryagin numbers for the bundles $X_c$, and show how one can find 
the $L_i$ from this.

\begin{lem}\label{l:pclasses}
The Pontryagin classes of a $\C P^{2k}$-bundle $X_c$ over $S^4$ are given by:
\begin{alignat*}{1}
p_1(X_c) &= (2k+1)y^2-2cx \ ,\\
p_2(X_c) &=k(2k+1)y^4-4c(k-1)xy^2 \ ,\\
&\cdots\\
p_{k+1}(X_c) &= {2k+1 \choose k+1}y^{2k+2} \ .
\end{alignat*}
\end{lem}
\begin{proof}
Given that the total Pontryagin class is
$$
p(TX_c) = (1+y^2)^{2k-1}(1+y_1^2)(1+y_2^2) 
$$
with $y_1+y_2=2y$ and $y_1y_2=y^2+cx$, we can calculate:
$$
p_1(X_c)=(2k-1)y^2+y_1^2+y_2^2=(2k-1)y^2+(y_1+y_2)^2-2y_1y_2 = (2k+1)y^2-2cx \ .
$$
Similarly
$$
p_2(X_c)={2k-1 \choose 2}y^4+(2k-1)y^2(y_1^2+y_2^2)+y_1^2y_2^2=\ldots = k(2k+1)y^4-4c(k-1)xy^2 \ ,
$$
and 
\begin{alignat*}{1}
p_{k+1}(X_c) &= {2k-1 \choose k+1}y^{2k+2}+{2k-1 \choose k}y^{2k}(y_1^2+y_2^2)+{2k-1 \choose k-1}y^{2k-2}(y_1y_2)^2\\
&={2k-1 \choose k+1}y^{2k+2}+{2k-1 \choose k}y^{2k}(2y^2-2cx)+{2k-1 \choose k-1}y^{2k-2}(y^2+cx)^2\\
&= \left( {2k-1 \choose k+1}+2{2k-1 \choose k}+{2k-1 \choose k-1}\right) y^{2k+2}\\
&= {2k+1 \choose k+1}y^{2k+2} \ .
\end{alignat*}
Here we have freely used the relations from Proposition~\ref{p:coho} in the calculations.
\end{proof}
As a consequence we can calculate Pontryagin numbers of $X_c$ explicitly.
\begin{lem}\label{l:p1}
$\langle p_1^{k+1}(TX_c),[X_c]\rangle = -(4k+3)(2k+1)^{k}\cdot c$
\end{lem}
\begin{proof}
Given $p_1(TX_c)=p_1(T\pi)=(2k+1)y^2 -2cx$, we compute using the relations in the cohomology ring from Proposition~\ref{p:coho}:
\begin{alignat*}{1}
p_1^{k+1}(TX_c) &=((2k+1)y^2 -2cx)^{k+1} = (2k+1)^{2k+2}y^{2k+2} - 2(k+1)(2k+1)^kcxy^{2k} \\
&= -(2k+1)^{k+1}cxy^{2k}  - 2(k+1)(2k+1)^kcxy^{2k} = -(4k+3)(2k+1)^{k} cxy^{2k} \ .
\end{alignat*}
The result follows since $\langle xy^{2k},[X_c]\rangle =1$.
\end{proof}
\begin{ex}
Setting $k=1$ in Lemma~\ref{l:p1}, we find $\langle p_1^2(TX_c),[X_c]\rangle = -21c$ for $\C P^2$-bundles over $S^4$. This was first
calculated by Dessai and Tuschmann~\cite[Proposition~2.2]{DT} using Lie theory.
\end{ex}

\begin{lem}\label{l:pk+1}
$\langle p_{k+1}(TX_c),[X_c]\rangle = - {2k+1 \choose k+1} \cdot c$
\end{lem}
\begin{proof}
Using the formula for $p_{k+1}(X_c)$ from Lemma~\ref{l:pclasses}, we have
$$
\langle p_{k+1}(X_c),[X_c]\rangle = {2k+1 \choose k+1} \langle y^{2k+2},[X_c]\rangle =-{2k+1 \choose k+1} \langle cxy^{2k},[X_c]\rangle = - {2k+1 \choose k+1} \cdot c \ .
$$
\end{proof}
Setting $k=1$ in this Lemma we find the second Pontryagin number $\langle p_2(TX_c),[X_c]\rangle = -3c$ for $\C P^2$-bundles over $S^4$. Since such a bundle 
has vanishing signature by Proposition~\ref{p:coho}, our calculations of $p_1^2$ and $p_2$ show that the signature must be a rational multiple of $7p_2-p_1^2$. 
Which multiple it is can be determined by computing for example for $\C P^2\times \C P^2$, which has signature one. Thus one finds the well known formula
$$
L_2 = \frac{1}{45}(7p_2-p_1^2) \ .
$$

To determine $L_3$ in dimension $12$ we need one more calculation.
\begin{lem}\label{l:p1p2}
$\langle p_{1}^{k-1}p_2(TX_c),[X_c]\rangle = - (2k+1)^{k-1}(4k^2+3k-4) \cdot c$
\end{lem}
Again the proof, which we omit, is a straightforward calculation using the formulae for $p_1$ and $p_2$ from Lemma~\ref{l:pclasses} and the relations in the 
cohomology ring from Proposition~\ref{p:coho}.

\begin{ex}
Using the last three Lemmata, a $\C P^{4}$-bundle $X_c$ over $S^4$ has Pontryagin numbers $p_1^3=-275c$, $p_1p_2=-90c$ and $p_3=-10c$.
\end{ex}
\begin{ex}
If we take the product of $\C P^2$ with a $\C P^2$-bundle over $S^4$, then using the Whitney sum formula for the 
Pontryagin classes of the tangent bundle and our calculation for the second factor, we find $p_1^3=-189c$, $p_1p_2=-72c$ and $p_3=-9c$.
\end{ex}
Now the only linear combinations of Pontryagin numbers that vanish for both of these examples are the multiples of $62p_3-13p_1p_2+2p_1^3$. Evaluating
this on $\C P^2\times\C P^2\times\C P^2$, which has signature one, gives
$$
L_3=\frac{1}{945}(62p_3-13p_1p_2+2p_1^3) \ .
$$

One can proceed into arbitrarily high dimensions by implementing the calculations involving~\eqref{eq:tp} and the relations from Proposition~\ref{p:coho} 
in an algorithm involving symbolic computation. The proof of Theorem~\ref{t:main} shows that the algorithm will determine $L_i$ unfailingly. For each dimension
the algorithm ends by solving a system of linear equations, and we have proved that the solution exists and is unique.

\bibliographystyle{amsplain}

\bigskip

\end{document}